\documentclass[a4paper,draft,reqno,12pt]{amsart}
\usepackage[english]{babel}
\usepackage{amsmath}
\usepackage{amssymb}
\usepackage{amscd}
\usepackage{amsthm}
\usepackage{euscript}
\usepackage{tikz}
\usepackage{pst-node}
\usepackage{tikz-cd} 
\newtheorem{proposition}{Proposition}
\newtheorem{lemma}{Lemma}
\newtheorem{theorem}{Theorem}

\theoremstyle{definition}
\newtheorem{definition}{Definition}

\theoremstyle{remark}
\newtheorem {remark}{Remark}

\def\Ker{{\rm Ker}\,}

\def\BK{{\mathbb K}}

\def\BK{{\mathbb K}}

\def\BN{{\mathbb N}}

\def\BA{{\mathbb A}}

\def\CO{\mathcal{O}}

\def\CI{\mathcal{I}}

\def\mf{\mathfrak{m}}

\title{On the generators of coordinate algebras of affine ind-varieties}

\author{Alexander Chernov}
\email{chenov2004@gmail.com}
\address{
Lomonosov Moscow State University, Faculty of Mechanics and Mathematics, Department of Higher Algebra, Leninskie Gory 1, Moscow, 119991 Russia;
\linebreak
and
\linebreak
HSE University, Faculty of Computer Science, Pokrovsky Boulevard 11, Moscow, 109028, Russia}
\begin{document}
\maketitle

\thispagestyle{empty}
\begin{abstract}
In this paper we study the structure of the coordinate ring of an affine ind-variety. We prove that any coordinate ring of an affine ind-variety which is not isomorphic to an affine algebraic variety doesn't have a countable set of generators. Also we prove that coordinate rings of affine ind-varieties have an everywhere dense subspace of countable dimension.
\end{abstract}

\section{Introduction}

Let $\BK$ be an algebraically closed field of characteristic zero. By an algebraic variety we always mean an algebraic variety over $\BK$. We will denote a countable product $\prod_{i = 1}^{\infty}\BK$ with a coordinate-wise multiplication by $\BK^{\infty}$.

\begin{definition}\cite[Definition 1.1.1]{FK}
    An \emph{ind-variety} $X$ is a set together with an ascending filtration $X_0 \subseteq X_1 \subseteq X_2 \subseteq \dotsc \subseteq X$ such that the following holds: \begin{enumerate}
        \item $X = \cup_{k\in \BN}X_k$;
        \item Each $X_k$ has the structure of an algebraic variety;
        \item For all $k \in \BN$, the inclusion $X_k \subseteq X_{k+1}$ is a closed immersion of algebraic varieties.
    \end{enumerate}
    An ind-variety is called \emph{affine ind-variety} if it admits a filtration such that all $X_k$ are affine. A \emph{morphism} between ind-varieties $X = \cup X_i$ and $Y = \cup Y_j$ is a map $\varphi: X \to Y$ such that for any $k$ there is an $l(k)$ such that $\varphi(X_k) \subseteq Y_l$ and the induced map $\varphi_{k}: X_k \to Y_{l(k)}$ is a morphism of varieties. An \emph{isomorphism} $\varphi$ of ind-varieties is a bijective morphism such that $\varphi^{-1}$ is also a morphism of ind-varieties.
\end{definition}

For an ind-variety $X = \cup X_i$ there is a projective system of coordinate algebras of $X_i$'s with obvious surjective projections $\pi_{km}:\CO(X_k)\to \CO(X_m)$ for $k \geq m$ arising from the inclusions $X_m \hookrightarrow X_k$. For a morphism of ind-varieties $\varphi: X \to Y$ there is a commutative diagram:

\[ \begin{tikzcd}[sep=large]
    \CO(X_m)  & \CO(X_k) \arrow[two heads]{l} \\%
    \CO(Y_{l(m)}) \arrow{u} & \CO(Y_{l(k)}) \arrow{u} \arrow[two heads]{l} 
    \end{tikzcd}
    \]

which yields a continuous homomorphism $\psi:\varprojlim\mathcal{O}(Y_l) \to \varprojlim\mathcal{O}(X_k)$ of topological algebras since both algebras have a standard topology of the inverse limit and $\psi$ is an inverse limit of mappings $\CO(Y_{l(k)}) \to \CO(X_k)$.

\begin{definition}\cite[Definition 1.1.4 (3)]{FK}
    The \emph{algebra of regular functions} on an affine ind-variety $X = \cup_{k\in\BN}X_k$ or the \emph{coordinate ring} is defined as \begin{equation*}
        \mathcal{O}(X) = \varprojlim\mathcal{O}(X_k)
    \end{equation*}
    The algebra $\mathcal{O}(X)$ is a topological algebra and it defines $X$ up to isomorphism. A topological algebra which is isomorphic to the coordinate ring of some affine ind-variety is called an \emph{affine ind-algebra}.
\end{definition}

The systematic study of ind-varieties and ind-groups was started by Shafarevich in \cite{SH1, SH2}. Later, this theory was further advanced by Kambayashi in addressing the Jacobian problem [4] and developing the framework for ind-affine schemes [5]. More recently, Furter and Kraft's work [1] provides a comprehensive overview of the contemporary study of ind-varieties and ind-groups.

Common algebraic varieties are also ind-varieties since they can be represented by a finite filtration. We introduce the following definition to separate this case.

\begin{definition}\label{strictdef}
    An ind-variety $X$ is called a \emph{strict ind-variety} if it is not isomorphic to a common algebraic variety i.e. it doesn't have a finite filtration by affine algebraic varieties.
\end{definition}

In our paper we prove that for any strict affine ind-variety its algebra of regular functions does not have a countable set of generators (Theorem~\ref{notcount}). But we show that any such algebra has a countable \emph{local basis} (see Definition~\ref{lb} and Theorem~\ref{countbasis}).

\section{Cardinality of the set of generators of an affine ind-algebra}

In the case of affine algebraic varieties their coordinate algebra is finitely generated and noetherian. The structure of the coordinate algebra of an affine ind-variety differs from the classical case.

\begin{lemma}\label{points}
Let $\varphi: X \to Y$ be a closed immersion of affine varieties and $\varphi^{*}: \mathcal{O}(Y) \to \mathcal{O}(X)$ is the corresponding homomorphism of algebras of regular functions. If $p \in Y \setminus X$ then $\varphi^{*}(\mf_{p}) = \mathcal{O}(X)$, where $\mf_{p}$ denotes the maximal ideal corresponding to the point $p$.
\end{lemma}
\begin{proof}
    Note that $\mf_{p}$ consists of functions vanishing at the point $p$ so $\varphi^{*}(\mf_{p}) = \{f \circ \varphi: f(p) = 0\}$. We can find a function $g \in \mathcal{O}(Y)$ such that $g{\big|}_X = 0$ but $g(p) \neq 0$. If $f \in \mathcal{O}(Y)$ then $f{\big|}_X = (f - \frac{f(p)}{g(p)}g){\big|}_X$ which is equivalent to $\varphi^{*}(f) = \varphi^{*}(f - \frac{f(p)}{g(p)}g)$. Since $\varphi^{*}$ is surjective and $f - \frac{f(p)}{g(p)}g \in \mf_{p}$ it implies that $\varphi^{*}(\mf_{p}) = \mathcal{O}(X)$.
    
\end{proof}

It turns out that there is a simple algebraic criterion for an affine ind-variety to be strict.

\begin{proposition}\label{strict}
    The following statements are equivalent:
    
    \begin{itemize}
        \item[a) ] $X$ is a strict affine ind-variety.
        \item[b) ] There is a surjective homomorphism $\mathcal{O}(X) \to \BK^{\infty}$.
    \end{itemize}
\end{proposition}

\begin{proof}
    a) $\Rightarrow$ b). We can represent $X$ by filtration $X = \bigcup X_i$, where all $X_i$ are affine algebraic varieties and the sequence \{$X_i$\} does not stabilize. Let $j_i: X_i \to X_{i+1}$ denote a closed immersion of the elements of this filtration. We can assume that $X_i \neq X_{i+1}$ so the inclusions are strict. Since the inclusions are strict we can find a set of points \{$p_i$\} such that $p_i \in X_i \setminus X_{i-1}$ so we have the following diagram:
    
    \[ \begin{tikzcd}[sep=large]
    \{p_1\} \arrow[hook]{r}{i_1} \arrow[hook]{d} & \{p_1, p_2\} \arrow[hook]{r}{i_2} \arrow[hook]{d} & \{p_1, p_2, p_3\} \arrow[hook]{r}{i_3} \arrow[hook]{d} & ...\\%
    X_1 \arrow[hook]{r}{j_1} & X_2 \arrow[hook]{r}{j_2} & X_3 \arrow[hook]{r}{j_3} & ...
    \end{tikzcd}
    \]
    
    Since the ideal of the union of distinct points is defined by the product of its ideals, this diagram yields the short exact sequence of projective systems:
    
    \[ \begin{tikzcd}[sep=large]
    0 & 0 & 0\\%
    \mathcal{O}(\{p_1\}) \arrow{u} & \mathcal{O}(\{p_1, p_2\}) \arrow[two heads]{l}{i_1^{*}} \arrow{u} & \mathcal{O}(\{p_1, p_2, p_3\}) \arrow[two heads]{l}{i_2^{*}} \arrow{u} & \arrow[two heads]{l}{i_3^{*}} ...\\%
    \mathcal{O}(X_1) \arrow[two heads]{u} & \mathcal{O}(X_2) \arrow[two heads]{u} \arrow[two heads]{l}{j_1^{*}} & \mathcal{O}(X_3) \arrow[two heads]{u} \arrow[two heads]{l}{j_2^{*}} & \arrow[two heads]{l}{j_3^{*}} ...\\%
    \mf_{p_1} \arrow[hook]{u} & \mf_{p_1}\mf_{p_2} \arrow[hook]{u} \arrow[two heads]{l}{j_1^{*}{\big|}_{\mf_{p_1}\mf_{p_2}}} & \mf_{p_1}\mf_{p_2}\mf_{p_3} \arrow[two heads]{l}{j_2^{*}{\big|}_{\mf_{p_1}\mf_{p_2}\mf_{p_3}}} \arrow[hook]{u} & \arrow[two heads]{l}{j_3^{*}{\big|}_{...}} ...\\%
    0 \arrow{u} & 0 \arrow{u} & 0 \arrow{u}
    \end{tikzcd}
    \]
    
    Note that the bottom projective system is correctly defined and it is surjective as for any $1 \leq i \leq k$ we have $j_k^{*}(\mf_{p_i}) = \mf_{p_i}$ and $j_k^{*}(\mf_{p_{k+1}}) = \mathcal{O}(X_k)$ by Lemma~\ref{points}.

    Since the bottom system of ideals is surjective, the limit functor preserves exactness (see \cite[Proposition 10.2]{AM} or \cite[Proposition 3.5.7]{W}) and we get the following short exact sequence:
    
    \[ \begin{tikzcd}[sep=normal]
    0 \arrow{r} & \varprojlim \mf_{p_1}...\mf_{p_n} \arrow{r} & \varprojlim \mathcal{O}(X_n) \arrow{r} \arrow[equal]{d} & \varprojlim \mathcal{O}(\{p_1, ..., p_n\}) \arrow{r} \arrow[equal]{d} & 0\\%
    & & \mathcal{O}(X) \arrow{r} & \mathcal{O}(\bigcup_{k \in \mathbb{N}} p_k) \arrow{r} & 0
    \end{tikzcd}
    \]
    
    Since $\mathcal{O}(\bigcup_{\BK \in \mathbb{N}} p_k) = \BK^{\infty}$, we get the required statement.
    \vspace{2mm}

    b) $\Rightarrow$ a). If $X$ is a common algebraic variety over $\BK$, then $\CO(X)$ is a noetherian ring. The algebra $\BK^{\infty}$ is not noetherian because the ideal of elements with finitely many non-zero components is not finitely generated. Since $\BK^{\infty}$ is not noetherian, it cannot be a surjective image of $\CO(X)$.

\end{proof}

The Proposition~\ref{strict} provides us with the main result of this section.

\begin{theorem}\label{notcount}
    Let X be a strict ind-variety. Then $\mathcal{O}(X)$ is non-noetherian and does not have a countable set of generators over $\BK$.
\end{theorem}
\begin{proof}
    It is easy to see that $\BK^{\infty}$ has the required properties. Since $\BK^{\infty}$ is a surjective image of $\mathcal{O}(X)$ for any $X$ by Proposition~\ref{strict}, it implies that $\mathcal{O}(X)$ is also non-noetherian and does not have a countable set of generators over $\BK$.
    
\end{proof}

\section{Local bases in affine ind-algebras}

Although the coordinate algebra of the affine ind-variety does not have a countable set of generators we can investigate its local structure. Let us introduce the following definition.

\begin{definition}\label{lb}
    Let $X$ be an affine ind-variety. We call a system $B \subseteq \mathcal{O}(X)$ a \emph{local basis} of the topological algebra $\mathcal{O}(X)$ if for any non-empty open subset $U \subseteq X$ there is an element $f \in U$ which is contained in the linear span of $B$.
\end{definition}

\begin{remark}
    The existence of a local basis is equivalent to the existence of an everywhere dense subspace of countable dimension.
\end{remark}

The main goal of this section is to prove the following theorem.

\begin{theorem}\label{countbasis}
    Let $X$ be an affine ind-variety. Then there is a countable local basis for the topological algebra $\mathcal{O}(X)$.
\end{theorem}

\begin{lemma}\cite[Lemma 1.5.2]{FK}\label{immersions}
    Let $\iota: X_1 \to X_2$ and $j_1: X_1 \to \BA^{n_1}$ be closed immersions of affine algebraic varieties. Then there are $n_2 \geq n_1$ and closed immersions $j_2: X_2 \to \BA^{n_2}$ and $i:\BA^{n_1} \to \BA^{n_2}$ such that $i \circ j_1 = j_2 \circ \iota$.
\end{lemma}

We shall denote the ind-variety $\cup_{n\in \BN} \BA^n$ by $\BA^{\infty}$. 

\begin{lemma}\label{spaceSur}
    Let $X$ be an affine ind-variety. Then there is a continuous surjective homomorphism $\CO(\BA^{\infty}) \to \CO(X)$.
\end{lemma}
\begin{proof}
    The proof is very similar to the proof of the Proposition~\ref{strict}. We can assume that $X$ has a strictly increasing filtration by affine algebraic varieties $X_k$, $k \geq 1$. Using the Lemma~\ref{immersions} we get the following commutative diagram of inclusions:

    \[ \begin{tikzcd}[sep=large]
    X_1 \arrow[hook]{r}{\iota_1} \arrow[hook]{d} & X_2 \arrow[hook]{r}{\iota_2} \arrow[hook]{d} & X_3 \arrow[hook]{r}{\iota_3} \arrow[hook]{d} & ...\\%
    \BA^{n_1} \arrow[hook]{r}{i_1} & \BA^{n_2} \arrow[hook]{r}{i_2} & \BA^{n_3} \arrow[hook]{r}{i_3} & ...
    \end{tikzcd}
    \]

    which yields the dual short exact sequence of projective systems:

    \[ \begin{tikzcd}[sep=large]
    0 & 0 & 0\\%
    \mathcal{O}(X_1) \arrow{u} & \mathcal{O}(X_2) \arrow[two heads]{l}{\iota_{1}^{*}} \arrow{u} & \mathcal{O}(X_3) \arrow[two heads]{l}{\iota_{2}^{*}} \arrow{u} & \arrow[two heads]{l}{\iota_{3}^{*}} ...\\%
    \mathcal{O}(\BA^{n_1}) \arrow[two heads]{u} & \mathcal{O}(\BA^{n_2}) \arrow[two heads]{u} \arrow[two heads]{l}{i_1^{*}} & \mathcal{O}(\BA^{n_3}) \arrow[two heads]{u} \arrow[two heads]{l}{i_2^{*}} & \arrow[two heads]{l}{i_3^{*}} ...\\%
    \CI(X_1) \arrow[hook]{u} & \CI(X_2) \arrow[hook]{u} \arrow[two heads]{l}{i_1^{*}{\big|}_{\CI(X_2)}} & \CI(X_3) \arrow[two heads]{l}{i_2^{*}{\big|}_{\CI(X_3)}} \arrow[hook]{u} & \arrow[two heads]{l}{i_3^{*}{\big|}_{\CI(X_4)}} ...\\%
    0 \arrow{u} & 0 \arrow{u} & 0 \arrow{u}
    \end{tikzcd}
    \]
    \vspace{2mm}

    where $\CI(X_i)$ denotes an ideal of $X_i$ in $\BA^{n_i}$.
    
    Since $\CI(X_{k+1})$ consists of functions vanishing along $X_{k+1}$, their restriction to any subset vanishes along $X_{k} \subset X_{k+1}$ but $i_k^{*}$ is exactly a restriction. It means that the bottom projective system is correctly defined and it is surjective. Just as in the proof of Proposition~\ref{strict} it implies that we have a continuous surjective homomorphism $\varprojlim \CO(\BA^{n_k}) \to \varprojlim\CO(X_k) = \CO(X)$. By the \cite[Proposition 1.4.8]{FK} $\BA^{\infty}$ does not depend on the filtration by affine spaces and then $\varprojlim \CO(\BA^{n_k}) \simeq \CO(\BA^{\infty})$ in our case. It completes the proof.
    
\end{proof}

\begin{proposition}\label{space}
    The algebra $\CO(\BA^{\infty})$ has a countable local basis.
\end{proposition}
\begin{proof}
    There is a canonical filtration $\BA^\infty = \cup_{k=1}^\infty \BA^k$ which yields us the canonical representation of $\mathcal{O}(\BA^\infty) = \BK^{[\infty]} = \varprojlim_{n}\BK[x_1, \dotsc, x_n]$ with projections:\begin{gather*}
        \pi_n^m:\BK[x_1,\dotsc,x_m] \to \BK[x_1, \dotsc,x_{n}], \ \pi_n^m(f(x_1, \dotsc,x_m)) = f(x_1,\dotsc,x_{n}, 0, \dotsc, 0), \ m \geq n\\
        \pi_n: \BK^{[\infty]} \to \BK[x_1.\dotsc,x_n], \ \pi_n(f(x_1, \dotsc, x_n, \dotsc)) = f(x_1,\dotsc,x_n,0,0,\dotsc)
    \end{gather*}
    Note that for every $\BK[x_1,\dotsc,x_n]$ there is a natural inclusion $i_n: \BK[x_1,\dotsc,x_n] \to \BK^{[\infty]}$ such that $\pi_n \circ i_n = id$. So we can assume that for every $f \in \BK^{[\infty]}$ an element $\pi_n(f)$ is also contained in $\BK^{[\infty]}$. The algebra $\BK[x_1,\dotsc,x_n]$ has an obvious countable basis \begin{equation*}
        B_n = \{x_1^{l_1}x_2^{l_2}\dotsc x_n^{l_n} \ | \ l_1,\dotsc,l_n \in \BN\cup\{0\}\}. 
    \end{equation*}

    $\BK^{[\infty]}$ has a topology induced by the countable set of ideals $I_n = \Ker \pi_n$ as a base of neighborhoods of 0. Consider the set $B = \cup_{n\in\BN}B_n$ in $\BK^{[\infty]}$. The set $B$ is countable since every $B_n$ is countable. For any $f \in \BK^{[\infty]}$ and for any $n \in \BN$ the image $\pi_n(f)$ is contained in the linear span of $B_n$ in $\BK[x_1,\dotsc,x_n]$. Hence $\pi_n(f)$ is contained in the linear span of $B$. \begin{equation*}
        g \in f + I_n  \ \Leftrightarrow \ \pi_n(f) = \pi_n(g)
    \end{equation*}
    For any neighborhood $U$ of $f \in \BK^{[\infty]}$ there is a number $n \in \BN$ such that $f + I_n \subseteq U$. Let $g = \pi_n(f) \in \BK^{[\infty]}$, it is contained in the linear span of $B$ and it is contained in the $U$ because it is contained in $f + I_n$ since $\pi_n(g) = g = \pi_n(f)$. It proves that the set $B$ is a local basis for $\mathcal{O}(\BA^{\infty)}$.

\end{proof}

Now we can prove the Theorem \ref{countbasis}.

\begin{proof}[Proof of the Theorem \ref{countbasis}]
    We use notations from the Proposition~\ref{space}. By Lemma~\ref{spaceSur} we have a continuous surjective homomorphism $\pi: \CO(\BA^{\infty}) \to \CO(X)$. It implies that $\pi(B)$ is a local basis for $\CO(X)$. Indeed, if an open subset $U \subset \CO(X)$ is non-empty, then its preimage in $\CO(\BA^{\infty})$ is non-empty due to surjectivity of $\pi$ and thus contains an element $f$ from the linear span of $B$. But it means that $\pi(f) \in U$ and $f$ is contained in the linear span of $\pi(B)$.

\end{proof}

\section*{Acknowledgments}

The author of this work is a scholarship holder of the Foundation for the Development of Theoretical Physics and Mathematics “BASIS”.

\end{document}